\documentclass[11pt,a4paper,twoside]{article}
\usepackage[T1]{fontenc}
\usepackage[english]{babel}
\usepackage[utf8]{inputenc}
\usepackage[pdftex]{graphicx}
\usepackage{amsmath,amsthm,amssymb,amsfonts,amsbsy}
\usepackage{lmodern}

\usepackage{color, colortbl} % colouring text
\usepackage{float} % determining exact positions of figures and tables using [H]
\usepackage{bookmark} % pdf bookmarks
\usepackage{microtype} % reducing the amount of "under-" and "overfull box" warnings
\usepackage[margin=2.5cm]{geometry} % setting page margins
\usepackage{booktabs} % typesetting tables (\toprule, \midrule, \bottomrule)
\usepackage{cite}
\usepackage{authblk} % The package redefines the \author command to work as normal or to allow a footnote style of author/affiliation input.
\usepackage{hyperref}

\theoremstyle{definition}

\theoremstyle{plain}
\newtheorem{te}{Theorem}[section]

\newtheorem{co}[te]{Corollary}

\newtheorem{conjecture}{Conjecture}[section]

\newcommand{\beq}{\begin{eqnarray}}
\newcommand{\eeq}{\end{eqnarray}}

\newcommand{\beqs}{\begin{eqnarray*}}
\newcommand{\eeqs}{\end{eqnarray*}}

\newcommand{\ds}{\displaystyle}
\allowdisplaybreaks

\newcommand{\ABC}{{\rm ABC}}
\usepackage{mathtools}

\title{Remarks on maximum atom-bond connectivity index \\ with given graph parameters }

\author{
    \large \bf Darko Dimitrov$^a$, Barbara Ikica$^b$, Riste {\v S}krekovski$^c$\footnote{Supported by the ARRS Program P1-0383.}
    }

\affil{  \normalsize
    $^a${ Hochschule f\"ur Technik und Wirtschaft Berlin, Germany
    }
    \\E-mail: {\tt darko.dimitrov11@gmail.com}
}

\affil{
    $^b${ Faculty of Mathematics and Physics, University of Ljubljana \& \\
    Institute of Mathematics, Physics, and Mechanics, Ljubljana, Slovenia}
    \\E-mail: {\tt barbara.ikica@fmf.uni-lj.si}
}

\affil{
    $^c${ Faculty of Information Studies, Novo mesto \& \\
    Faculty of Mathematics and Physics, University of Ljubljana \&  \\ Faculty of Mathematics, Natural Sciences and Information Technologies, \\University of Primorska, Koper, Slovenia}
    \\E-mail: {\tt skrekovski@gmail.com}
}

\date{\today}

\begin{document}

\maketitle

\begin{abstract}
The atom-bond connectivity (ABC) index is a degree-based molecular structure descriptor
that can be used for modelling thermodynamic properties of organic chemical compounds.
Motivated by its applicable potential, a series of investigations have been carried out in the past several years.
In this note we first consider graphs with given edge-connectivity that attain the maximum ABC index.
In particular, we give an affirmative answer to the conjecture about the structure of graphs with edge-connectivity equal to one that maximize the ABC index, which was recently raised by Zhang, Yang, Wang and Zhang~\cite{zywz-mabciggp-2016}.
In addition, we provide supporting evidence for another conjecture posed by the same authors which concerns graphs that maximize the ABC index among all graphs with chromatic number equal to some fixed $\chi \geq 3$. Specifically, we confirm this conjecture in the case where the order of the graph is divisible by $\chi$.
\end{abstract}
%
%
%
% ------------------------------------------------------------------------------------
%-------------------    Section one:  INTRODUCTION and related results      ----------
% ------------------------------------------------------------------------------------
%
%
%
% -------------------------------------------------------------------------
% ----------------------           bibliography       ---------------------
% -------------------------------------------------------------------------
%
%------------------------------------------------------------------------------
%
\section[Introduction]{Introduction}

Let $G$ be a simple undirected graph of order $n=|V(G)|$ and size $m=|E(G)|$.
For $v \in V(G)$, let $d_G(v)$ denote the degree of $v$, that is, the number of edges incident
to $v$.
The {\em atom-bond connectivity (ABC) index} is defined as
\beq \label{eqn:001}
\ABC(G)=\sum_{uv\in E(G)} f(d(u), d(v)), \nonumber
\eeq
where  $\ds f(d(u), d(v)) =\sqrt{\frac{d(u) +d(v)-2}{d(u)d(v)}}$.

The ABC index was introduced in 1998 by Estrada, Torres, Rodr{\' i}guez and Gutman~\cite{etrg-abc-98}, who
showed that this index correlates well with the heats of formation of alkanes and can therefore serve the purpose of predicting their thermodynamic properties.
Various physico-chemical applications of the ABC index were presented in a few other works, including~\cite{e-abceba-08, gtrm-abcica-12}.
These results triggered a number of mathematical and computational investigations into the ABC index and its extension~\cite{adgh-dctmabci-14, bcpt-nubfta-16,
cg-eabcig-11, clg-subabcig-12, d-abcig-10, dt-cbfgaiabci-10,
dmga-cbabcig-16, d-ectmabci-2013, d-sptmabci-2014, d-sptmabci-2-2015, ddf-sptmabci-3-2016,  dis-rggibg-2016, df-ftmabc-2015,
dgf-abci-11, f-abcvgga-2016, fgiv-cstmabci-12, gfahsz-abcic-2013, gmng-abcitfnl-15, gs-tsaiotwnpv-16, hag-ktmabci-14, lccglc-fcstmaibtds-14,
lmccgc-pstmabci-15,  lcmzczj-otwmaiatwgnol-16, p-rubabci-14, vh-mabcict-2012, xzd-frabcit-2010,xzd-abcicg-2011}. For recent advances in chemical graph theory, one may consult the following two surveys \cite{kst-maowi-2016, aks-maof-2016}.

Recently, in~\cite{zywz-mabciggp-2016} the maximum ABC index across all connected graphs of a given order, with a
fixed independence number, number of pendant vertices, chromatic number or edge-connectivity was considered.
There, among other results, an upper bound and a characterization of graphs with a given edge-connectivity $\geq 2$ that attain the maximum ABC index was given. Before we state this result, we need the following definitions:

\begin{itemize}
\item Given two disjoint graphs $G$ and $H$ with disjoint vertex sets $V(G)$ and $V(H)$ and disjoint edge sets $E(G)$ and $E(H)$, the \emph{disjoint union} of $G$ and $H$ is the graph
$G + H$ with vertex set $V(G) \cup V(H)$ and edge set $E(G) \cup E(H)$.

\item The \emph{ join} of two simple undirected graphs $G$ and $H$ is the graph $G \vee H$ with vertex set $V(G \vee H) = V (G)\cup V (H)$ and edge set $E(G \vee H)=E(G)\cup E(H)\cup \left\{uv : u \in V(G), \; v \in V(H) \right\}$.

\item The \emph{chromatic number} $\chi(G)$ of a graph $G$ is the smallest number of colours needed to colour the vertices of $G$ in such a way that no two adjacent vertices are coloured the same.
\end{itemize}

Let $K_n(k)$ denote the graph $K_k \vee (K_1+K_{n-k-1})$.
Observe that the graph $K_n(k)$ is simply a graph obtained by joining one vertex to $k$ vertices in $K_{n-1}$.
An illustration of  $K_6(3)$ is depicted in Figure~\ref{fig-K63}.

\begin{figure}[h]
\begin{center}
\includegraphics[scale=1.0]{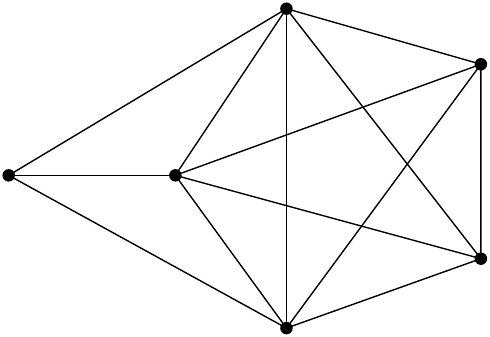}
\caption{ The graph $K_6(3)$.} \label{fig-K63}
\end{center}
\end{figure}

Let $T_{n,l}$ denote a complete $l$-partite graph of order $n$ with $|t_i-t_j| \leq 1$, where $t_i$, $i = 1, 2, \dots, l$, is the number of vertices in the $i$-th partition set of $T_{n,l}$.

The following result about graphs with prescribed edge-connectivity that attain the maximum ABC index was
obtained by Zhang, Yang, Wang and Zhang in~\cite{zywz-mabciggp-2016}.

\begin{te} \label{thm-edge-conn-org}
Let $G$ be a connected graph on $n \geq 6$ vertices with edge-connectivity
$k \geq 2$. Then,
\begin{eqnarray} %\label{thm-10}
\ABC(G) & \leq &k \sqrt{\frac{n+k-3}{k(n-1)}} + \frac{k(k-1)}{2(n-1)}\sqrt{2n-4}
+\frac{(n-k-1)(n-k-2)}{2(n-2)}\sqrt{2n-6} + \nonumber \\
&  & k(n-k-1) \sqrt{\frac{2n-5}{(n-1)(n-2)}},  \nonumber
\end{eqnarray}
with equality if and only if $G \cong K_n(k)$.
\end{te}

In~\cite{zywz-mabciggp-2016} it was conjectured that
Theorem~\ref{thm-edge-conn-org} also holds for $k=1$.
In addition, in the same paper, the following theorem and a conjecture related to it were posed.

\begin{te} \label{thm-chrom-org}
For any connected graph G of order n with chromatic number $\chi=2$:
\begin{itemize}
\item if $n$ is even, then $\ABC(G) \leq \frac{n}{2} \sqrt{n-2}$ with equality if and only if
$G \cong T_{n,2}$;
\item if $n$ is odd, then $\ABC(G) \leq \frac{1}{2} \sqrt{(n-2)(n^2-1)}$ with equality if and only if
$G \cong T_{n,2}$.
\end{itemize}
\end{te}

\begin{conjecture}\label{conj-chrom}
Let $G$ be an $n$-vertex connected graph with chromatic number $\chi \geq 3$. Then,
    \begin{equation*}
        \ABC(G) \leq \ABC(T_{n,\chi}),
    \end{equation*}
with equality if and only if $G \cong T_{n,\chi}$.
\end{conjecture}

Before we affirm the conjecture related to Theorem~\ref{thm-edge-conn-org} and
prove a special case of Conjecture~\ref{conj-chrom},
we state a few auxiliary results.
Namely, adding an edge to a graph strictly increases its ABC index~\cite{dgf-abci-11}
or, equivalently, removing an edge from a graph strictly decreases it~\cite{cg-eabcig-11}.
This has the following immediate consequence.

\begin{co} \label{corollary-Kn}
Among all connected  graphs with $n$ vertices, the maximum value of the $\ABC$ index is attained exactly by the complete graph $K_n$.
\end{co}

The subsequent result plays an important role in the next section, where we prove the first of the aforementioned conjectures.

\begin{te} {\em{[Karamata's inequality]}} \label{thm:karamata}
Let $U \subseteq \mathbb{R}$  be an open interval and $f : U \rightarrow U$ a convex function. Let $a_1 \geq a_2 \geq \cdots \geq a_n$ and $b_1 \geq b_2 \geq \cdots \geq b_n$ be such elements in $U$ that inequalities $a_1 + a_2 + \cdots +a_i  \geq b_1 + b_2 + \cdots +b_i $ hold
for every $i \in \{1, 2, \dots, n \}$ and equality holds for $i = n$. Then, $f(a_1) + f(a_2) + \cdots + f(a_n) \geq f(b_1) + f(b_2) + \cdots + f(b_n )$.
\end{te}

\section[Graphs with edge-connectivity one that maximize the ABC index]{Graphs with edge-connectivity one that maximize the ABC index}

\begin{te} \label{te-max-abc-connectivity}
Let $G$ be a graph with the maximum atom-bond connectivity index among all graphs with $n$ vertices and edge-connectivity $1$.
Then, $G \cong K_n(1)$.
%Then $G \cong K_1 \vee (K_1+K_{n-2})$.
\end{te}

\begin{proof}
Let $e=uv$ be an edge of $G$ whose deletion disconnects $G$ into two connected
components with $x$ and $y=n-x$ vertices, respectively.
Since $G$ maximizes the ABC index, it follows by Corollary~\ref{corollary-Kn} that
these two components must be complete subgraphs $K_x$ and $K_{y}$.
Let $u \in V(K_x)$ and $v \in V(K_{y})$.
Due to symmetry, we may assume that the possible values of $x$  are $1, 2, \dots, \lfloor n/2 \rfloor$.

If $x=1$, then the theorem holds.
If $x=2$, then $K_x$ is comprised of one edge, denoted here by $wu$.
Now, we add edges between $u$ and every vertex of $V(K_{y}) \setminus \{v\}$, thus obtaining a graph $G'$
which is comprised of $K_{1}$ and $K_{n-1}$ and the edge  $wu$ that connects them.
Corollary~\ref{corollary-Kn} implies that $\ABC(G') >\ABC(G)$.

For $x \geq 3$ it holds that
\beq \label{eq-thm-con-10}
\ABC(G) &=& \binom{x-1}{2}f(x-1,x-1) + (x-1)f(x-1,x) + f(x,y)  \nonumber \\
&&+ \binom{y-1}{2}f(y-1,y-1) + (y-1)f(y-1,y). \nonumber
\eeq

Let $w$ be a vertex in $V(K_x)$ different from $u$.
Let $G'$ be a graph obtained from $G$ by deleting all edges adjacent to $w$  and adding edges between $w$ and each vertex in $V(K_y)$.
It follows that
\beq \label{eq-thm-con-20}
\ABC(G') &=& \binom{x-2}{2}f(x-2,x-2) + (x-2)f(x-2,x-1) + f(x-1,y+1)  \nonumber \\
&&+ \binom{y}{2}f(y,y) + yf(y,y+1). \nonumber
\eeq

For $z > 1$, let
\beq %\label{eq-thm-con-30}
s(z) = \binom{z}{2}f(z,z) + z f(z,z+1).  \nonumber
\eeq
Next, we show that the function $s(z)$ is convex.
After simplification, the second derivative of $s(z)$ reads

\beq \label{eq-thm-con-30}
s''(z)=\frac{1}{8} \left(\frac{3 \sqrt{2}}{\sqrt{z-1}}+\frac{24}{(z+1)^{5/2} \sqrt{z (2 z-1)}}-\frac{2 (1-2 z (z+2))^2}{(z+1)^{5/2} (z (2 z-1))^{3/2}}\right). \nonumber
\eeq
The expression $24/((z+1)^{5/2} \sqrt{z (2 z-1)})$ is positive for $z > 1$. The remainder is positive as well, which can be shown by comparing the square of the minuend to the square of the subtrahend (notice that both the minuend and the subtrahend are positive). More precisely, after clearing denominators, we need to show that
\begin{equation*}
    18 (z+1)^5 (2z^2 - z)^3 > 4 (1 - 2z^2 - 4z)^4 (z-1)
\end{equation*}
holds for all $z > 1$. Clearly, $(z+1) > (z-1)$ is satisfied, and the inequality $(2z^2 - z) > (1 - 2z^2 - 4z)$ follows immediately from
\begin{equation*}
    2z^2 - z - 1 + 2z^2 + 4z = 4z^2 + 3z - 1 = (4z - 1)(z + 1) > 0.
\end{equation*}
It now suffices to show that $(z+1)^2 > (1 - 2z^2 - 4z)$ holds. Since
\begin{equation*}
    z^2 + 2z + 1 - 1 + 2z^2 + 4z = 3z^2 + 6z > 0,
\end{equation*}
we finally conclude that $s''(z)$ is positive, and thus that $s(z)$ is a convex function.

Consider now the difference $\ABC(G')- \ABC(G)$. It holds that
\beq \label{eq-thm-con-40}
\ABC(G')- \ABC(G)&=& s(x-2) + f(x-1,y+1)+ s(y) \nonumber \\
 && -s(x-1) - f(x,y) - s(y-1)  \nonumber.
\eeq

Recall that the range of values of $x$ is $1, 2, \dots, \lfloor n/2 \rfloor$, which implies  $y \geq x$. Under this constraint, a straightforward verification shows that
\beq \label{eq-thm-con-90}
f(x-1,y+1) >  f(x,y).
\eeq

Since $y > y-1$ and $y +(x-2) = (y-1) + (x-1)$ and the function $s$ is convex, it follows by Theorem~\ref{thm:karamata} that
\beq \label{eq-thm-con-100}
 s(y) + s(x-2) \geq s(y-1) + s(x-1).
\eeq

Inequalities (\ref{eq-thm-con-90}) and (\ref{eq-thm-con-100}) imply the validity of $\ABC(G')- \ABC(G) >0$,
which concludes the proof of the theorem.

\end{proof}

As the class of $k$-vertex-connected graphs is a subclass of the class of $k$-edge-connected graphs where $k \geq 1$, and as $K_n(k)$ is a $k$-vertex-connected graph, we may conclude from Theorems~\ref{thm-edge-conn-org}~and~\ref{te-max-abc-connectivity} that a graph that maximizes the $\ABC$ index among all $k$-edge-connected graphs also maximizes the $\ABC$ index among all $k$-vertex-connected graphs.

\begin{co} \label{te-max-abc-vertex-connectivity}
Let $G$ be a graph with the maximum atom-bond connectivity index among all graphs with $n$ vertices and vertex-connectivity $k \geq 1$.
%Then $G \cong K_1 \vee (K_1+K_{n-2})$.
Then, $G \cong K_n(k)$.
\end{co}

\section[Graphs with chromatic number $\chi$ that maximize the ABC index]{Graphs with chromatic number $\chi$ that maximize the ABC index}

In this section, we prove a special case of Conjecture~\ref{conj-chrom}. Specifically, we show that this conjecture holds when $\chi$ divides $n$.

\begin{te} \label{thm2-chromatic}
Let $G$ be an $n$-vertex connected graph with chromatic number $\chi \geq 2$ and suppose that $\chi$ divides $n$. Then,
    \begin{equation*}
        \ABC(G) \leq \ABC(T_{n,\chi}),
    \end{equation*}
with equality if and only if $G \cong T_{n,\chi}$.
\end{te}

\begin{proof}
As the case $\chi = 2$ has already been proven in \cite{zywz-mabciggp-2016}, we may assume from now on that $\chi \geq 3$.

Consider a $\chi$-colouring of $G$ and denote by $t_i$ the size of the $i$-th colour class, that is, the number of vertices that are assigned the $i$-th colour for $i = 1, 2, \dots, \chi$.

First of all, since adding an edge to a graph strictly increases its $\ABC$ index, notice that
\begin{equation}
\label{eq:ABCuppBnd}
    \ABC(G) \leq \sum_{i<j} t_i t_j \sqrt{\frac{2n-t_i-t_j-2}{(n-t_i)(n-t_j)}} = \sum_{i<j} \frac{t_i t_j}{\sqrt{n-t_i}\sqrt{n-t_j}} \sqrt{2n-(t_i+t_j)-2}.
\end{equation}
Write $x_{ij}$ and $y_{ij}$ as an abbreviation for the first factor and the second factor on the right-hand side, respectively, i.e.,
\begin{equation*}
    x_{ij} = \frac{t_i t_j}{\sqrt{n-t_i}\sqrt{n-t_j}} \quad \text{and} \quad y_{ij} = \sqrt{2n-(t_i+t_j)-2}.
\end{equation*}

Observe now that the right-hand side of \eqref{eq:ABCuppBnd} is equivalent to the standard scalar product $\langle x, y \rangle$ of vectors $x := (x_{ij})_{i<j}$ and $y := (y_{ij})_{i<j}$ of length ${\chi \choose 2}$. Thus, by the Cauchy-Schwarz inequality,
\begin{equation}
\label{eq:ABC-CS-ineq}
    \ABC(G) \leq | \langle x, y \rangle | \leq \|x\| \|y\|.
\end{equation}

In the following, we show that this upper bound on $\ABC(G)$ is attained exactly whenever $t_i = n/\chi$ for all $i = 1, 2, \dots, \chi$ and every vertex has degree $n - n/\chi$ (this follows from maximizing the number of edges); that is, by $T_{n,\chi}$, as claimed.

Note that the equality $|\langle x, y \rangle | = \|x\| \|y\|$ in \eqref{eq:ABC-CS-ineq} holds if and only if $x$ and $y$ are linearly dependent, that is, if and only if $x = \mu y$ holds for some scalar $\mu \neq 0$.
Hence, if and only if
\begin{equation*}
    \mu = x_{ij}/y_{ij} = \frac{t_i t_j}{\sqrt{n-t_i}\sqrt{n-t_j}} \left( \sqrt{2n-(t_i+t_j)-2} \right)^{-1}
\end{equation*}
for each pair of indices $i < j$.

Take indices $i, j, k$ such that $1 \leq i < j < k \leq \chi$. From $\mu = x_{ij}/y_{ij} = x_{ik}/y_{ik}$ it follows that
\begin{equation*}
    t_j^2 (n - t_k) \big( 2n - (t_i + t_k) -2 \big) = t_k^2 (n - t_j) \big( 2n - (t_i + t_j) - 2 \big).
\end{equation*}
One finds that $t_j = t_k$ must hold. If not, then $t_j > t_k$ (or vice versa), which implies $n - t_k > n - t_j$ and $2n - t_i - t_k - 2 > 2n - t_i - t_j - 2$. Hence, as these terms are all non-zero, the left-hand side is strictly greater than the right-hand side, which leads to a contradiction.

Since $i$, $j$ and $k$ were chosen arbitrarily, this establishes that $t_1 = t_2 = \cdots = t_{\chi}$ and thus, since $\sum_{i=1}^{\chi} t_i = n$, it follows that $t_i = n/\chi$ for all $i = 1, 2, \ldots, \chi$, as desired.
\end{proof}

As a consequence of  Theorem~\ref{thm2-chromatic} we obtain the following corollary.

\begin{co}
Let $G$ be an $n$-vertex connected graph with chromatic number $\chi \geq 2$ and suppose that $\chi$ divides $n$. Then,
    \begin{equation*}
        \ABC(G) \leq  n \sqrt{\frac{\chi(n-1)-n}{2\chi}},
    \end{equation*}
with equality if and only if $G \cong T_{n,\chi}$.
\end{co}

\begin{proof}
Let us now compute the upper bound \eqref{eq:ABC-CS-ineq} on $\ABC(G)$. The calculation of $\|y\|$ is straightforward (and its value is independent of $t_i$):
\begin{equation*}
    \|y\|^2 = \sum_{i<j} (2n-(t_i+t_j)-2) = 2n {\chi \choose 2} - (\chi-1) n - 2 {\chi \choose 2} = (\chi-1) \big( \chi(n-1) - n \big).
\end{equation*}
We proceed by evaluating the expression $\|x\|$ at $t_i = n/\chi$:
\begin{equation*}
    \|x\|^2 = \frac{n^2}{2 \chi (\chi-1)}.
\end{equation*}
Thus, combining these statements and plugging them in \eqref{eq:ABC-CS-ineq}, we find that
\begin{equation*}
    \ABC^2(G) \leq \frac{(\chi-1) \big( \chi(n-1) - n \big) n^2}{2\chi(\chi-1)} = \frac{n^2 \big( \chi(n-1) - n \big)}{2\chi}.
\end{equation*}
Notice that this upper bound is attained if and only if $G \cong T_{n,\chi}$, which concludes the proof.
\end{proof}

\section{Concluding comments} \label{sec:conclusion}

In this note we consider graphs with given edge-connectivity that attain the maximum ABC index.
The case where the connectivity is greater or equal to two was solved in \cite{zywz-mabciggp-2016}.
Here we resolve the remaining case where edge-connectivity equals one.
The authors of \cite{zywz-mabciggp-2016} also posed a conjecture about the structure of graphs with chromatic number equal to some fixed $\chi \geq 3$ that maximize the ABC index.
Herein we confirm a special case of this conjecture, more specifically, the case where the order of the graph is divisible by $\chi$.

\end{document}